\documentclass[11pt]{amsart}
\usepackage{eurosym}
\usepackage{graphicx,amscd,color,amsmath,amsfonts,amssymb,geometry}

\newtheorem{theorem}{Theorem}[section]

\newtheorem{corollary}[theorem]{Corollary}

\numberwithin{equation}{section}

\begin{document}
\title[The optimal constants of the mixed $\left( \ell _{1},\ell _{2}\right)
$-Littlewood inequality]{The optimal constants of the mixed $\left( \ell
_{1},\ell _{2}\right) $-Littlewood inequality}
\author[D. Pellegrino ]{Daniel Pellegrino}
\address{Departamento de Matem\'{a}tica \\
\indent Universidade Federal da Para\'{\i}ba \\
\indent 58.051-900 - Jo\~{a}o Pessoa, Brazil.}
\email{pellegrino@pq.cnpq.br and dmpellegrino@gmail.com}
\thanks{The author is supported by CNPq Grant 401735/2013-3 - PVE - Linha 2
and INCT-Matem\'{a}tica.}
\maketitle

\begin{abstract}
In this note, among other results, we find the optimal constants of the
generalized Bohnenblust--Hille inequality for $m$-linear forms over $\mathbb{%
R}$ and with multiple exponents $\left( 1,2,...,2\right) $, sometimes called
mixed $\left( \ell _{1},\ell _{2}\right) $-Littlewood inequality. We show
that these optimal constants are precisely $\left( \sqrt{2}\right) ^{m-1}$
and this is somewhat surprising since a series of recent papers have shown
that similar constants have a sublinear growth. This result answers a
question raised by Albuquerque \textit{et al.} in a paper published in 2014
in the \textit{Journal of Functional Analysis}.
\end{abstract}

\section{Introduction}

In the recent years a lot of papers (see, for instance, \cite{bohr, ann,
diniz, Nuuu} and the references therein) have been dedicated to the search
of best (or even optimal constants) for a class of famous inequalities,
including the Littlewood's $4/3$ inequality, the Bohnenblust--Hille
inequality and the multilinear Hardy--Littlewood inequality (see \cite{bh,
hardy, LLL}). The search of these constants, besides its intrinsic interest,
have been shown to be important in different areas of Mathematics and even
in Physics (see \cite{ann, monta}). In this paper we find the optimal
constants of a class of inequalities that encompasses the sometimes called
mixed $\left( \ell _{1},\ell _{2}\right) $-Littlewood inequality, which
plays an important role in the recent development of the theory related to
the Bohnenblust--Hille inequality.

\bigskip\ The Khinchine inequality (see \cite{Di}) asserts that for all $%
0<p<\infty $, there exist positive constants $A_{p}$ and $B_{p}$ such that
\begin{equation}
A_{p}\left( \sum_{n=1}^{N}\left\vert a_{n}\right\vert ^{2}\right) ^{\frac{1}{%
2}}\leq \left( \int_{0}^{1}\left\vert \sum_{n=1}^{N}a_{n}r_{n}\left(
t\right) \right\vert ^{p}dt\right) ^{\frac{1}{p}}\leq B_{p}\left(
\sum_{n=1}^{N}\left\vert a_{n}\right\vert ^{2}\right) ^{\frac{1}{2}}
\label{lpo}
\end{equation}%
for every positive integer $N$ and all real scalars $a_{1},...,a_{N}$ (here,
$r_{n}$ denotes the $n$-$th$ Rademacher function, which is defined in $[0,1]$
by $r_{n}(t)=sgn\left( \sin 2^{n+1}\pi t\right) $)$.$

The optimal constants of the Khinchine inequality are known. It is simple to
observe that the optimal value of $A_{p}$ is $1$ for all $p\geq 2$ and $%
B_{p}=1$ for all $p\leq 2.$ For real scalars, U. Haagerup (\cite{haage})
proved that the optimal constants $A_{p}$ are (see also \cite[page 23]{Di})
\begin{equation}
A_{p}=\frac{1}{\sqrt{2}}\left( \frac{\Gamma \left( \frac{p+1}{2}\right) }{%
\sqrt{\pi }}\right) ^{\frac{1}{p}},\ \ \text{ for }1.85\approx p_{0}<p<2
\label{kinreal}
\end{equation}%
and
\begin{equation}
A_{p}=2^{\frac{1}{2}-\frac{1}{p}},\ \ \text{ for }1\leq p\leq p_{0}\approx
1.85.  \label{kinreal2}
\end{equation}%
The exact definition of $p_{0}$ is the following: $p_{0}\in (1,2)$ is the
unique real number satisfying
\begin{equation*}
\Gamma \left( \frac{p_{0}+1}{2}\right) =\frac{\sqrt{\pi }}{2}.
\end{equation*}%
\bigskip

Note that the Khinchine inequality tells us that
\begin{equation}
\left( \int_{0}^{1}\left\vert \sum_{n=1}^{N}a_{n}r_{n}\left( t\right)
\right\vert ^{p}dt\right) ^{\frac{1}{p}}\leq B_{p}A_{r}^{-1}\left(
\int_{0}^{1}\left\vert \sum_{n=1}^{N}a_{n}r_{n}\left( t\right) \right\vert
^{r}dt\right) ^{\frac{1}{r}}  \label{kcc}
\end{equation}%
regardless of the $0<p,r<\infty .$ From now on, as usual, $c_{0}$ denotes
Banach space, endowed with the $\sup $ norm, of the sequences of scalars
converging to zero. If $U:c_{0}\times c_{0}\rightarrow \mathbb{R}$ is a
bilinear form, from the Khinchine inequality (and noting that from (\ref%
{kinreal2}) we have $A_{1}=2^{-1/2}$) we have, for all positive integers $N$,%
\begin{eqnarray*}
\sum\limits_{i=1}^{N}\left( \sum\limits_{j=1}^{N}\left\vert
U(e_{i},e_{j})\right\vert ^{2}\right) ^{\frac{1}{2}} &\leq &\sqrt{2}%
\sum\limits_{i=1}^{N}\int\limits_{0}^{1}\left\vert
\sum\limits_{j=1}^{N}r_{j}(t)U(e_{i},e_{j})\right\vert dt \\
&= &\sqrt{2}\int\limits_{0}^{1}\sum\limits_{i=1}^{N}\left\vert
U(e_{i},\sum\limits_{j=1}^{N}r_{j}(t)e_{j})\right\vert dt \\
&\leq &\sqrt{2}\sup_{t\in \left[ 0,1\right] }\sum\limits_{i=1}^{N}\left\vert
U(e_{i},\sum\limits_{j=1}^{N}r_{j}(t)e_{j})\right\vert \\
&\leq &\sqrt{2}\left\Vert U\right\Vert \sup_{t\in \left[ 0,1\right]
}\left\Vert \sum\limits_{j=1}^{N}r_{j}(t)e_{j}\right\Vert \\
&=&\sqrt{2}\left\Vert U\right\Vert .
\end{eqnarray*}

The above argument is part of the proof of the famous Littlewood's $4/3$
inequality \cite{LLL}, in a modern terminology (see \cite{garling}). The
exponent $4/3$ in Littlewood's $4/3$ inequality is a kind of interpolation
of mixed $\left( \ell _{1},\ell _{2}\right) $ and $\left( \ell _{2},\ell
_{1}\right) $ sums.

\bigskip By using a by now well known argument, one can generalize
Khinchine's inequality to multiple sums as follows (see, for instance \cite%
{bohr}): let $0<p<\infty $, and let $%
(a_{i_{1},...,i_{m}})_{i_{1},...,i_{m}=1}^{N}$ be a matrix of real scalars.
Then%
\begin{align*}
\left( A_{p}\right) ^{m}\left( \sum\limits_{i_{1},...,i_{m}=1}^{N}\left\vert
a_{i_{1}...i_{m}}\right\vert ^{2}\right) ^{1/2}& \leq \left(
\int\nolimits_{\lbrack 0,1]^{m}}\left\vert
\sum%
\limits_{i_{1},...,i_{m}=1}^{N}r_{i_{1}}(t_{1})...r_{i_{m}}(t_{m})a_{i_{1}...i_{m}}\right\vert ^{p}dt_{1}...dt_{m}\right) ^{1/p}
\\
& \leq \left( B_{p}\right) ^{m}\left(
\sum\limits_{i_{1},...,i_{m}=1}^{N}\left\vert a_{i_{1}...i_{m}}\right\vert
^{2}\right) ^{1/2}
\end{align*}%
for all positive integers $N$. Using this \textquotedblleft multiple
Khinchine inequality\textquotedblright , with some effort the mixed $\left(
\ell _{1},\ell _{2}\right) $-Littlewood inequality is obtained (see, for
instance, \cite{alb}, \cite[Lemma 2]{deff} or the proof of \cite[Proposition
3.1]{bohr}).

\begin{theorem}[Mixed $\left( \ell _{1},\ell _{2}\right) $-Littlewood
inequality]
For all real $m$-linear forms $U:c_{0}\times \cdots \times c_{0}\rightarrow
\mathbb{R}$ we have%
\begin{equation*}
\sum_{j_{1}=1}^{N}\left( \sum_{j_{2},...,j_{m}=1}^{N}\left\vert
U(e_{j_{1}},...,e_{j_{m}})\right\vert ^{2}\right) ^{\frac{1}{2}}\leq \left(
\sqrt{2}\right) ^{m-1}\left\Vert U\right\Vert
\end{equation*}%
for all positive integers $N.$
\end{theorem}

\bigskip In this note, among other results, we prove that the constants $%
\left( \sqrt{2}\right) ^{m-1}$ are sharp. This answers a question raised in
\cite{alb} and closes the possibility that the optimal constants associated
to the generalized Bohnenblust--Hille inequality have all sublinear growth
(see Section \ref{ytr}), contradicting what could be at a first glance
predicted from a series of recent related results (see \cite{alb, n, ap,
bohr, diniz, Nuuu}).

\bigskip

\section{The optimal constants of the mixed $\left( \ell _{1},\ell
_{2}\right) $-Littlewood inequality}

\bigskip \bigskip In \cite[Remark 5.1]{alb} the following sentence raises a
natural question: \textquotedblleft It would be nice to know if the
constants in the mixed $\left( \ell _{1},\ell _{2}\right) $-Littlewood
inequality can also be chosen to be subpolynomial.\textquotedblright\ Our
next result shows that these constants can not be subpolynomial for real
scalars, since we prove that the optimal values are $\left( \sqrt{2}\right)
^{m-1}.$ This is somewhat surprising since various recent papers (see \cite%
{alb, n, ap, bohr, diniz, Nuuu}) have shown that similar constants in the
theory of multilinear forms (see details in the next section) have a
subpolynomial (sublinear) growth.

\begin{theorem}
\label{kjh}\bigskip The optimal constants of mixed $\left( \ell _{1},\ell
_{2}\right) $-Littlewood inequality for real $m$-linear forms are $\left(
\sqrt{2}\right) ^{m-1}.$
\end{theorem}

\begin{proof}
Let us denote the optimal constants of the mixed $\left( \ell _{1},\ell
_{2}\right) $-Littlewood inequality for real $m$-linear forms by $%
C_{m,\infty }$. From the previous section we know that the optimal constants
are not bigger than $\left( \sqrt{2}\right) ^{m-1}.$ So it suffices to show
that $\left( \sqrt{2}\right) ^{m-1}$ is a lower bound. The proof is done by
induction and follows the lines of the proof of the main result of \cite%
{diniz}.

For $m=2$, let $T_{2}:c_{0}\times c_{0}\rightarrow \mathbb{R}$ be defined by
\begin{equation*}
T_{2}(x,y)=x_{1}y_{1}+x_{1}y_{2}+x_{2}y_{1}-x_{2}y_{2}.
\end{equation*}%
The signal minus before $x_{2}y_{2}$ is strategic to make the norm of $T_{2}$
small (for our purposes). It is not difficult to prove that $\left\Vert
T_{2}\right\Vert =2.$ Since
\begin{equation*}
\sum\limits_{i_{1}=1}^{2}\left( \sum\limits_{i_{2}=1}^{2}\left\vert
T_{2}(e_{i_{^{1}}},e_{i_{2}})\right\vert ^{2}\right) ^{\frac{1}{2}}=2\sqrt{2}
\end{equation*}%
we conclude that%
\begin{equation*}
C_{2,\infty }\geq \frac{2\sqrt{2}}{2}=\left( \sqrt{2}\right) ^{2-1}.
\end{equation*}

For $m=3$ let $T_{3}:c_{0}\times c_{0}\times c_{0}\rightarrow \mathbb{R}$ be
given by
\begin{equation*}
T_{3}(x,y,z)=(z_{1}+z_{2})\left(
x_{1}y_{1}+x_{1}y_{2}+x_{2}y_{1}-x_{2}y_{2}\right) +(z_{1}-z_{2})\left(
x_{3}y_{3}+x_{3}y_{4}+x_{4}y_{3}-x_{4}y_{4}\right)
\end{equation*}%
and note that $\Vert T_{3}\Vert =4$. Since
\begin{equation*}
\sum\limits_{i_{1}=1}^{4}\left( \sum\limits_{i_{2},i_{3}=1}^{4}\left\vert
T_{3}(e_{i_{^{1}}},e_{i_{2}},e_{i_{3}})\right\vert ^{2}\right) ^{\frac{1}{2}%
}=4\sqrt{4}
\end{equation*}%
we have
\begin{equation*}
C_{3,\infty }\geq \frac{4\sqrt{4}}{4}=\left( \sqrt{2}\right) ^{3-1}.
\end{equation*}%
In the case $m=4$ we consider $T_{4}:c_{0}\times c_{0}\times c_{0}\times
c_{0}\rightarrow \mathbb{R}$ given by
\begin{align*}
& T_{4}(x,y,z,w)= \\
& =\left( w_{1}+w_{2}\right) \left(
\begin{array}{c}
(z_{1}+z_{2})\left( x_{1}y_{1}+x_{1}y_{2}+x_{2}y_{1}-x_{2}y_{2}\right) \\
+(z_{1}-z_{2})\left( x_{3}y_{3}+x_{3}y_{4}+x_{4}y_{3}-x_{4}y_{4}\right)%
\end{array}%
\right) \\
& +\left( w_{1}-w_{2}\right) \left(
\begin{array}{c}
(z_{3}+z_{4})\left( x_{5}y_{5}+x_{5}y_{6}+x_{6}y_{5}-x_{6}y_{6}\right) \\
+(z_{3}-z_{4})\left( x_{7}y_{7}+x_{7}y_{8}+x_{8}y_{7}-x_{8}y_{8}\right) .%
\end{array}%
\right)
\end{align*}%
and a similar argument shows that $\left\Vert T_{4}\right\Vert =8$ and
\begin{equation*}
\sum\limits_{i_{1}=1}^{8}\left(
\sum\limits_{i_{2},i_{3},i_{4}=1}^{8}\left\vert
T_{4}(e_{i_{^{1}}},e_{i_{2}},e_{i_{3}},e_{i_{4}})\right\vert ^{2}\right) ^{%
\frac{1}{2}}=8\sqrt{8}.
\end{equation*}%
Therefore
\begin{equation*}
C_{4,\infty }\geq \frac{8\sqrt{8}}{8}=\left( \sqrt{2}\right) ^{4-1}.
\end{equation*}%
The general case is proved by induction. Define the $m$-linear forms $%
T_{m}:c_{0}\times \overset{m}{\ldots }\times c_{0}\rightarrow {\mathbb{R}}$
by induction as
\begin{align*}
T_{m}(x_{1},\ldots ,x_{m})=& (x_{m}^{1}+x_{m}^{2})T_{m-1}(x_{1},\ldots
,x_{m-1}) \\
&
+(x_{m}^{1}-x_{m}^{2})T_{m-1}(B^{2^{m-2}}(x_{1}),B^{2^{m-2}}(x_{2}),B^{2^{m-3}}(x_{3})\ldots ,B^{2}(x_{m-1})),
\end{align*}%
where $x_{k}=(x_{k}^{n})_{n}\in c_{0}$ for $1\leq k\leq m$, $1\leq n\leq
2^{m-1}$ and $B$ is the backward shift operator in $c_{0}$. Then $\Vert
T_{m}\Vert =2^{m-1}$ for all $m\in {\mathbb{N}}$ and the proof follows
straightforwardly.
\end{proof}

\section{The general Bohnenblust--Hille inequality\label{ytr}}

\bigskip The mixed $\left( \ell _{1},\ell _{2}\right) $-Littlewood
inequality is a particular instance of the generalized Bohnenblust--Hille
inequality for real scalars:

\textbf{Theorem (Generalized Bohnenblust--Hille inequality for real scalars,
(\cite{alb}, 2014)). }\textit{Let }$m\geq 2$ be a positive integer, \textit{%
and} $\mathbf{q}:=(q_{1},...,q_{m})\in \left[ 1,2\right] ^{m}$ be \textit{%
such that}%
\begin{equation*}
\frac{1}{q_{1}}+\cdots +\frac{1}{q_{m}}=\frac{m+1}{2}.
\end{equation*}%
\textit{Then there exists a constant} $C_{m,\infty ,\mathbf{q}}\geq 1$
\textit{such that}%
\begin{equation}
\left( \sum_{j_{1}=1}^{N}\left( \sum_{j_{2}=1}^{N}\left( \cdots \left(
\sum_{j_{m}=1}^{N}\left\vert T(e_{j_{1}},...,e_{j_{m}})\right\vert
^{q_{m}}\right) ^{\frac{q_{m-1}}{q_{m}}}\cdots \right) ^{\frac{q_{2}}{q_{3}}%
}\right) ^{\frac{q_{1}}{q_{2}}}\right) ^{\frac{1}{q_{1}}}\leq C_{m,\infty ,%
\mathbf{q}}\left\Vert T\right\Vert  \label{gbh}
\end{equation}%
\textit{for all continuous }$m$\textit{--linear forms }$T:c_{0}\times \cdots
\times c_{0}\rightarrow \mathbb{R}$\textit{\ and all positive integers }$N.$%
\textit{\ }

Henceforth $\mathbf{q}:=(q_{1},...,q_{m})$ will be called a multiple
exponent. \medskip A natural question is whether the approach of the
previous section can give the optimal constants for other multiple exponents
of the generalized Bohnenblust--Hille inequality. We do not know the answer
but, if compared to the best known upper bounds, the values are different,
in general. The case of the classical exponents is more evident since the
lower bounds for the constants associated to $\left( \frac{2m}{m+1},...,%
\frac{2m}{m+1}\right) $ with real scalars are $2^{1-\frac{1}{m}}$ \ (and
thus not bigger than $2$) while the best known upper bounds are of the order
$m^{0.36482}$ (see \cite{bohr}).\

The same argument of the proof of Theorem \ref{kjh} furnishes the following
theorem:

\begin{theorem}
Let $\alpha \in \lbrack 1,2]$ be a constant and $\mathbf{q}=(\alpha ,\beta
_{m},...,\beta _{m})$ be a multiple exponent of the generalized
Bohnenblust--Hille inequality for real scalars. Then
\begin{equation*}
C_{m,\infty ,\mathbf{q}}\geq 2^{\frac{2m-\alpha m-4+3\alpha }{2\alpha }}
\end{equation*}
\end{theorem}

\begin{proof}
Note that%
\begin{equation*}
\beta _{m}=\frac{2\alpha m-2\alpha }{\alpha m-2+\alpha }.
\end{equation*}%
Plugging the $m$-linear forms defined in the proof of Theorem \ref{kjh} into
the Bohnenblust--Hille inequality we obtain%
\begin{equation*}
C_{m,\infty ,\mathbf{q}}\geq \frac{\left( 2^{m-1}\left( 2^{m-1}\right) ^{%
\frac{\alpha m-2+\alpha }{2\alpha m-2\alpha }\cdot \alpha }\right) ^{\frac{1%
}{\alpha }}}{2^{m-1}}=2^{\frac{2m-\alpha m-4+3\alpha }{2\alpha }}.
\end{equation*}
\end{proof}

Note that when $\alpha <2$ the above theorem shows that the respective
optimal constants have an exponential growth. The following corollary shows
that the result is sharp in the sense that when $\alpha =2$ the growth is
sublinear.

\begin{corollary}
Let $\alpha \in \lbrack 1,2]$ be a constant and $\mathbf{q}=(\alpha ,\beta
_{m},...,\beta _{m})$ be a multiple exponent of the generalized
Bohnenblust--Hille inequality for real scalars. Then the optimal constants
associated to $\mathbf{q}$ have an exponential growth if and only if $\alpha
<2.$
\end{corollary}

\begin{proof}
The case $\alpha <2$ is done in the above theorem (note that the growth of
the optimal constants can not be bigger than exponential because from \cite%
{alb} we know that $C_{m,\infty ,\mathbf{q}}\leq \left( \sqrt{2}\right)
^{m-1}$ regardless of the $\mathbf{q.}$

If $\alpha =2$ we obtain $\mathbf{q}=(2,\frac{2m-2}{m},...,\frac{2m-2}{m}),$
and in this case the optimal constants have a sublinear growth. In fact,
from the proof of \cite[Proposition 3.1]{bohr}, using the optimal constants
of the Khinchine inequality we conclude that the optimal constants
associated to $(\frac{2m-2}{m},...,\frac{2m-2}{m},2)$ have a sublinear
growth; this is a by now classic consequence of the Khinchine inequality. By
using the Minkowski inequality as in \cite{alb, n} we can move the number $2$
to the first position and conclude that the constants associated to $\mathbf{%
q}=(2,\frac{2m-2}{m},...,\frac{2m-2}{m})$ are dominated by the constants
associated to $(\frac{2m-2}{m},...,\frac{2m-2}{m},2)$, and thus have a
sublinear growth.
\end{proof}

\textbf{Acknowledgment.} The author thanks the referee for the corrections
that improved the final presentation of this paper.

\end{document}